\documentclass[a4paper, bibliography=totoc]{scrartcl}
\usepackage{amsmath}
\usepackage{amssymb}
\usepackage{enumerate}
\usepackage{amsthm}
\usepackage{todonotes}
\usepackage{mathtools} % for DeclarePairedDelimiter
\usepackage[UKenglish]{babel}
\usepackage[T1]{fontenc}
\usepackage[utf8]{inputenc}
\usepackage{tikz}
\usepackage{thmtools} %needed for thm-restate package
\usepackage{thm-restate} %to allow restating a theorem

\usepackage{hyperref} % Usually should be declared as the last package
\definecolor{blue3}{rgb}{.1,.0,.4}
\hypersetup{colorlinks=true, linkcolor=red, urlcolor=blue3, citecolor=blue3, pdfpagemode=UseNone, pdfstartview=FitH, bookmarksopen=true} %\usepackage{spelling}
\usepackage[open]{bookmark} %overrides warnings caused by the fact that I skip some elements of the hierarchy \section, \subsection, \subsubsection, \paragraph

\declaretheorem[name=Theorem,numberwithin=section]{thm} %Weird thing, the package thmtools ``requires'' it.

\newtheorem*{define*}{Definition}
\newtheorem{define}[thm]{Definition}

\newtheorem*{lemma*}{Lemma}
\newtheorem{lemma}[define]{Lemma}

\newtheorem*{algorithm*}{Algorithm}

\newtheorem*{construction*}{Construction}

\newtheorem*{prop*}{Proposition}

\newtheorem*{obs*}{Observation}

\newtheorem*{fact*}{Fact}

\newtheorem*{remark*}{Remark}

\newtheorem*{quest*}{Question}

\newtheorem*{cor*}{Corollary}

\newtheorem*{conjecture*}{Conjecture}

\newtheorem*{question*}{Question}

\newtheorem*{example*}{Example}

\newtheorem{convention}[define]{Convention}

\numberwithin{claimcounter}{define}
\newtheorem*{claim*}{Claim}

\numberwithin{equation}{section}

\newcommand{\Jac}{\operatorname{Jac}}

\newcommand{\R}{\mathbb{R}}

\newcommand{\Z}{\mathbb{Z}}
\newcommand{\lip}{\operatorname{Lip}}

\newcommand{\N}{\mathbb{N}}

\newcommand{\id}{\operatorname{id}}

\newcommand{\leb}{\mathcal{L}}

\DeclareMathOperator{\diam}{diam}

\newcommand{\mb}[1]{\mathbf{#1}}

\newcommand{\mc}[1]{\mathcal{#1}}

\newcommand{\abs}[1]{\left|#1\right|}
\newcommand{\lnorm}[2]{\left\|#2\right\|_#1}

\newcommand{\Prob}{\mathbb{P}}

\newcommand{\set}[1]{\left\{#1\right\}}
\newcommand{\cl}[1]{\overline{#1}}
\DeclarePairedDelimiter{\floor}{\lfloor}{\rfloor}

\def\XXint#1#2#3{{\setbox0=\hbox{$#1{#2#3}{\int}$ }
		\vcenter{\hbox{$#2#3$ }}\kern-.6\wd0}}

\makeatletter
\newcommand{\mylabel}[2]{#2\def\@currentlabel{#2}\label{#1}}
\def\blfootnote{\xdef\@thefnmark{}\@footnotetext} %footnote without a label
\makeatother

\title{Lipschitz constant $\log n$\\ almost surely suffices for\\mapping $n$ grid points onto a cube.}
\author{Michael Dymond}
\begin{document}
	\maketitle
	\begin{abstract}
		Kalu\v za, Kopeck\'a and the author show in \cite{DKK2018}, that the best Lipschitz constant for mappings taking a given $n^{d}$-element set in the integer lattice $\Z^{d}$, with $n\in \N$, surjectively to the regular $n$ times $n$ grid $\set{1,\ldots,n}^{d}$ may be arbitrarily large. However, there remain no known, non-trivial asymptotic bounds, either from above or below, on how this best Lipschitz constant grows with $n$. We approach this problem from a probabilistic point of view. More precisely, we consider the random configuration of $n^{d}$ points inside a given finite lattice and establish almost sure, asymptotic upper bounds of order $\log n$ on the best Lipschitz constant of mappings taking this set surjectively to the regular~$n$~times~$n$~grid~$\set{1,\ldots,n}^{d}$.
	\end{abstract}
\blfootnote{The author acknowledges the support of Austrian Science Fund (FWF): P 30902-N35.}

\section{Introduction}
Fix a dimension $d\geq 2$. For each $n\in\N$, we define a mapping $F_{n}\colon \binom{\Z^{d}}{n^{d}}\to (0,\infty)$ by
\begin{equation*}
	F_{n}(S)=\min\set{\lip(f)\colon f\colon S\to \set{1,\ldots,n}^{d}\text{ surjective}}, \qquad S\in \binom{\Z^{d}}{n^{d}}.
\end{equation*}
We use the notation $\binom{X}{m}$ to denote the set of all $m$-element subsets of a set $X$.

The quantity $F_{n}(S)$ can be thought of as a quantification of how much the set $S$ differs from the regular $n\times n$ grid. In the 1990's Feige asked the question~\cite[Question 2.12]{Matousek_open} of whether the sequence
\begin{equation}\label{eq:Feige_seq}
	\mb{F}_{n}:=\sup_{S\in \binom{\Z^{d}}{n^{d}}}F_{n}(S),\qquad n\in \N,
\end{equation}
is bounded. In other words, Feige's question asks whether there is some absolute constant $L>0$ so that for any $n\in\N$ and any set $S\in \binom{\Z^{d}}{n^{d}}$ there exists a bijection $S\to [n]^{d}$ with Lipschitz constant at most $L$. Feige's question is motivated by his work~\cite{Feige_approx} on the bandwidth problem in computer science.

The paper~\cite{DKK2018} of Kalu\v{z}a, Kopeck\'a and the author provides a negative answer to Feige's question, proving that $\limsup_{n\to\infty}\mb{F}_{n}=\infty$. However, \cite{DKK2018} fails to impose any non-trivial asymptotic bounds on the Feige sequence $(\mb{F}_{n})$; the only available inequality of note is the trivial upper bound $\mb{F}_{n}\leq \sqrt{d}\cdot n$. This provides the motivation for the present work. 

In order to prove that $\limsup_{n\to\infty}\mb{F}_{n}=\infty$, \cite{DKK2018} uses a strategy introduced by McMullen~\cite{McM} and Burago and Kleiner~\cite{BK1} which allows for the translation of discrete Lipschitz problems to the continuous setting. The aforementioned authors developed this strategy in order to answer a long standing open question of Gromov~\cite{Grom}, namely whether every two separated nets in Euclidean space are bilipschitz equivalent. McMullen~\cite{McM} and Burago and Kleiner~\cite{BK1} introduce methods of encoding measurable density functions $\rho\colon [0,1]^{d}\to (0,\infty)$ as separable nets in $\R^{d}$ and use these to prove that Gromov's question about separated nets is equivalent to the question of whether every bounded density $\rho\colon [0,1]^{d}\to (0,\infty)$ admits a bilipschitz solution $f\colon [0,1]^{d}\to \R^{d}$ to the pushforward equation
\begin{equation}\label{eq:pushforward}
	f_{\sharp}\rho\leb|_{[0,1]^{d}}=\leb|_{f([0,1]^{d})}.
\end{equation}
McMullen~\cite{McM} and Burago and Kleiner~\cite{BK1} then resolve Gromov's question negatively by constructing $\rho$ for which \eqref{eq:pushforward} has no bilipschitz solutions. For the negative answer to Feige's question, Kalu\v za, Kopeck\'a and the author constructed $\rho$ so that \eqref{eq:pushforward} additionally has no solutions in the larger class of Lipschitz mappings $f\colon [0,1]^{d}\to \R^{d}$. Moreover, in a recent paper \cite{DymondKaluza2019}, Kalu\v za and the author find densities $\rho$ for which \eqref{eq:pushforward} has no solutions in the class of homeomorphisms $f\colon [0,1]^{d}\to \R^{d}$ for which both $f$ and $f^{-1}$ have modulus of continuity bounded above by $\omega(t)=t\log\left(\frac{1}{t}\right)^{\varphi_{0}(d)}$, where $\varphi_{0}(d)\to 0$ as $d\to\infty$.

Assuming a connection between the question of the asymptotic growth of the sequence $(\mb{F}_{n})$ and the continuous question of existence of solutions $f$ to \eqref{eq:pushforward} with prescribed modulus of continuity $\omega$, the latter result may hint towards an asymptotic lower bound of the form 
\begin{equation}\label{eq:asymp_lower_bound_feige}
	\mb{F}_{n}\geq (\log n)^{\varphi_{0}(d)}
\end{equation}
on the Feige sequence $(\mb{F}_{n})$. More precisely, it seems natural to conjecture that the inequality $\mb{F}_{n}\geq n\omega\left(\frac{1}{n}\right)$ holds asymptotically for moduli of continuity $\omega$ for which $\omega$-continuous solutions $f$ to \eqref{eq:pushforward} may be excluded.

In this note we identify certain types of sets $S\in\binom{\Z^{d}}{n^{d}}$ for which we are able to provide a non-trivial asymptotic upper bound on $F_{n}(S)$. Furthermore, we show that these sets occur with high probability, in a sense to be made precise shortly. We hope that this could be a step towards establishing bounds on the Feige sequence $(\mb{F}_{n})$. Note that the latter requires bounding $F_{n}(S)$ for a general set $S\in \binom{\Z^{d}}{n^{d}}$.

To determine the $n$-th Feige number $\mb{F}_{n}$, observe that it suffices to consider only sets $S\in \binom{\Z^{d}}{n^{d}}$ which lie inside the finite cubic grid $\set{1,\ldots,n^{d}}^{d}$ of side length $n^{d}$. Put differently, the supremum in \eqref{eq:Feige_seq} remains unchanged if the integer lattice $\Z^{d}$ is replaced by the finite grid $\set{1,\ldots,n^{d}}^{d}$. This holds because any set $S\in \binom{\Z^{d}}{n^{d}}$ may be mapped via a $1$-Lipschitz, injective mapping to a subset of $\set{1,\ldots,n^{d}}^{d}$: simply take out empty hyperplanes, contract and translate. Thus, to establish asymptotic bounds on the Feige sequence $(\mb{F}_{n})$ it suffices to provide asymptotic bounds on $F_{n}(S)$ for sets $S\in \binom{\set{1,\ldots,n^{d}}^{d}}{n^{d}}$. 

Restricting our attention to configurations of $n^{d}$ points inside a finite cubic grid, instead of inside the entire integer lattice, naturally invites a probabilistic approach. We can think of each possible configuration of the $n^{d}$ points in the finite cubic grid as occurring with equal probability.  Taking large cubic grids, such as the grid $\set{1,\ldots,n^{d}}^{d}$ discussed above, we would expect to see configurations of $n^{d}$ points being very spread out with high probability, leading to $F_{n}$ being uniformly bounded independent of $n$ with high probability. Thus, it makes sense to consider the problem in all smaller cubic grids of side length $\floor{cn}$ for all $c>1$. Inside grids $\set{1,\ldots,\floor{c_{n}n}}^{d}$ given by a sequence of numbers $c_{n}>1$ which does not converge to $1$ too quickly, we will prove an upper bound on $F_{n}$ which holds asymptotically almost surely. Note that for all sequences $(c_{n})$ with $1<\inf c_{n}\leq \sup c_{n} <\infty$, the methods of \cite{DKK2018} establish that there are sequences $(S_{n})\in \prod_{n=1}^{\infty} \binom{\set{1,\ldots,\floor{c_{n}\cdot n}}^{d}}{n^{d}} $ for which $\limsup_{n\to\infty} F_{n}(S_{n})=\infty$. There is less known in  the case of sequences $(c_{n})$ converging to $1$. It is not known for which sequences $c_{n}\searrow 1$ the sequence space $\prod_{n=1}^{\infty} \binom{\set{1,\ldots,\floor{c_{n}\cdot n}}^{d}}{n^{d}}$ contains a sequence $(S_{n})$ such that $\limsup_{n\to \infty}F_{n}(S_{n})=\infty$.

The notion of `asymptotically almost surely' refers to the uniform probability measures on the spaces $\binom{\set{1,\ldots,\floor{c_{n}\cdot n}}^{d}}{n^{d}}$ for $n\in\N$. In the present work we only consider one type of probability space, namely that given by a finite set equipped with the uniform probability measure. If $X$ is a finite, non-empty set, we consider the uniform probability measure on $(X,2^{X})$ defined by 
\begin{equation}\label{eq:P_X}
	\Prob_{X}(A)=\frac{\abs{A}}{\abs{X}}, \qquad A\subseteq X,
\end{equation}
where $\abs{-}$ denotes the cardinality. Since it will always be clear from the context which probability space we are working in, we will always just write $\Prob$ (without a subscript) to denote the uniform probability measure.

We are now ready to state the main result:
\begin{restatable}{thm}{aas}\label{thm:aas}
	Let $d\in\N$ with $d\geq 2$ and $q\in \R$ with $q\geq 1$ and $q>\frac{3}{d}$. For each $n\in\N$ let $c_{n}\geq \left(1+\frac{2^{d+7}}{\log n}\right)^{1/d}$, $\Omega_{n}:=\binom{\set{1,\ldots,\floor{c_{n}\cdot n}}^{d}}{n^{d}}$ and consider the probability space $(\Omega_{n},2^{\Omega_{n}},\Prob=\Prob_{\Omega_{n}})$ defined by \eqref{eq:P_X} and the random variable $F_{n}\colon \Omega_{n}\to(0,\infty)$ defined by
	\begin{equation*}
		F_{n}(S)=\min\set{\lip(f)\colon f\colon S\to \set{1,\ldots,n}^{d}\text{ surjective}}, \qquad S\in \Omega_{n}.
	\end{equation*}  Then there exists a constant $\Gamma=\Gamma(d)>0$ such that 
	\begin{equation*}
		\Prob\left[F_{n}\leq \Gamma (\log n)^{q}\right]\geq 1-\Gamma n^{\Gamma}\exp\left(-\frac{(\log n)^{qd-2}}{\Gamma}\right).
	\end{equation*}
	In particular, we have that
	\begin{equation*}
		\lim_{n\to\infty}\Prob\left[F_{n}\leq \Gamma (\log n)^{q}\right]=1.
	\end{equation*}
\end{restatable}
Theorem~\ref{thm:aas} tentatively supports the conjecture of  a polylogarithmic upper bound on the Feige sequence $(\mathbf{F}_{n})$. This is interesting because it coincides in form with the conjectured lower bound \eqref{eq:asymp_lower_bound_feige}, coming from the completely independent results of \cite{DymondKaluza2019}. 

\section{Preliminaries, Convention and Notation.}\label{sec:prelim_notation}
Let us quickly summarise some basic notation which may not be completely standard. The dimension $d$ of the Euclidean space $\R^{d}$ in which we work will be considered fixed throughout the whole paper. Thus, many objects defined in the paper should be thought of as having a suppressed subindex $d$; for example $F_{n}=F_{n,d}$. For a set $A$ and $k\in\N$ we let $\binom{A}{k}$ denote the set of all subsets of $A$ with precisely $k$ elements. Given $t\geq 0$ we write $\floor{t}$ for the integer part of $t$ and $[t]$ for the set of integers $\set{1,2,\ldots,\floor{t}}$. Since powers of $2$ arise frequently in the calculations we take, for convenience, the logarithm function $\log$ with base $2$. We will also write $\exp(x)$ to denote $2^{x}$.
\begin{convention}\label{conv:randomset}
	Let $X$ be a finite set, $N\in \N$, $\dagger$ stand for an abstract property and $\alpha\geq 0$. In Sections~\ref{sec:random} and \ref{sec:proof}, we write that a random set $S\in \binom{X}{N}$ satisfies
	\begin{equation*}
	\Prob[\text{$S$ has property $\dagger$}]\leq \alpha
	\end{equation*}
	as shorthand for the statement
	\begin{equation*}
	\Prob\left(\set{S\in \binom{X}{N}\colon \text{$S$ has property $\dagger$}}\right)\leq \alpha
	\end{equation*}
	in the probability space $\left(\binom{X}{N},2^{\binom{X}{N}},\Prob=\Prob_{\binom{X}{N}}\right)$ given by \eqref{eq:P_X}.
\end{convention}
The symbol $\leb$ will denote the Lebesgue measure. Furthermore, given a measurable function $\rho\colon [0,1]^{d}\to[0,\infty)$ we denote by $\rho\leb$ the measure on $[0,1]^{d}$ defined by $\rho\leb(A)=\int_{A}\rho\,d\leb$. If $\mu$ is a measure on $[0,1]^{d}$ and $f\colon [0,1]^{d}\to \R^{d}$ is a mapping, we will denote by $f_{\sharp}\mu$ the measure on $f([0,1]^{d})$ defined by $f_{\sharp}\mu(A)=\mu({f^{-1}(A)})$. The closure of a set $E$ will be written as $\cl{E}$. 

Given a measure $\mu$ on $[0,1]^{d}$, we call a collection  $\mc{T}$ of $\mu$-measurable subsets of $[0,1]^{d}$ a $\mu$-partition of $[0,1]^{d}$ if $\mu\left([0,1]^{d}\setminus \bigcup \mc{T}\right)=0$ and $\mu(T\cap T')=0$ for all $T,T'\in\mc{T}$ with $T\neq T'$. 
For each $k\in\N$ we let
\begin{equation}\label{def:Tk}
	\mc{T}_{k}=\set{ \prod_{i=1}^{d}\left(\frac{p_{i}}{k},\frac{p_{i}+1}{k}\right]\colon p_{1},\ldots,p_{d}\in \set{0,1,\ldots,k-1}},
\end{equation}
Note that each $\mc{T}_{k}$ is, in particular, an $\leb$-partition of $[0,1]^{d}$.

The next lemma is our main mechanism for relating measures to the question of best Lipschitz constants for mappings of finite sets.
\begin{lemma}\label{lemma:cts_to_discrete}
	Let $\mu,\nu$ be Borel probability measures on the unit cube $[0,1]^{d}$. Let $n\in\N$, $\mc{T}$ be a finite $\mu$-partition of $[0,1]^{d}$, $c>1$ and 
	\begin{equation*}
		X\subseteq \frac{1}{c n}\Z^{d}\cap \bigcup \mc{T}, \qquad Y\subseteq \frac{1}{n}\Z^{d}\cap [0,1]^{d}
	\end{equation*}
	be finite sets such that 
	\begin{align}
		\mu(T)&\geq \frac{1}{n^{d}}\abs{X\cap T} ,\qquad \text{ for every $T\in\mc{T}$, and}\label{lower_bound_mu}\\
		\nu(E)&\leq\frac{1}{n^{d}}\abs{\set{y\in Y\colon d_{\infty}(y,E)\leq \frac{1}{n}}}\qquad \text{for every $\nu$-measurable $E\subseteq [0,1]^{d}$,} \label{upper_bound_nu}
	\end{align}
	where $d_{\infty}$ denotes the distance induced by the norm $\lnorm{\infty}{-}$. Let $f\colon [0,1]^{d}\to [0,1]^{d}$ be a Lipschitz mapping with $f_{\sharp}\mu=\nu$. Then there exist a constant $\Lambda=\Lambda(d)>0$ and an injective mapping $g\colon X\to Y$ with
	\begin{equation*}
		\lip(g)\leq \Lambda\max\set{1,\lip(f)}c \left(n\cdot\max_{T\in\mc{T}}\diam T+1\right).
	\end{equation*}
\end{lemma}
\begin{proof}
	For a point $x\in X$ we denote by $T(x)$  a choice of set $T\in \mc{T}$ which contains $x$. We further define a set valued mapping $R\colon X\to 2^{Y}$ by
	\begin{equation*}
		R(x)=\left\{y\in Y \colon d_{\infty}(y,f(T(x)))\leq\frac{1}{n}\right\}.
	\end{equation*}
	In what follows we obtain an injective mapping $g\colon X\to Y$ with the property that
	\begin{equation}\label{eq:f_tilda_cond}
		g(x)\in R(x),\qquad x\in X.
	\end{equation}
	We may then complete the proof in the following way. For distinct points $x,x'\in X$ we observe that
	\begin{equation*}
		\lnorm{2}{g(x')-g(x)}\leq\lnorm{2}{g(x')-f(x')}+\lnorm{2}{f(x')-f(x)}+\lnorm{2}{f(x)-g(x)}
	\end{equation*}
	Now, from condition~\eqref{eq:f_tilda_cond} we have 
	\begin{align*}
		\lnorm{\infty}{g-f|_{X}}&\leq
		\max_{T\in \mc{T}}\diam f(T)+\frac{\sqrt{d}}{n}\\
		&\leq\lip(f)\sqrt{d}\max_{T\in\mc{T}}\diam T+\frac{\sqrt{d}}{n}\\
		&\leq \sqrt{d}\max\set{1,\lip(f)}\left(\max_{T\in\mc{T}}\diam T+\frac{1}{n}\right).
	\end{align*}
	Hence, using $\lnorm{2}{x'-x}\geq\frac{1}{c n}$, we obtain
	\begin{align*}
		\lnorm{2}{g(x')-g(x)}&\leq 2\sqrt{d}\max\set{1,\lip(f)}\left(\max_{T\in\mc{T}}\diam T+\frac{1}{n}\right)+\lip(f)\lnorm{2}{x'-x}\\
		&\leq 3\sqrt{d}\max\set{1,\lip(f)}c\left(n\cdot\max_{T\in\mc{T}}\diam T+1\right)\lnorm{2}{x'-x}.
	\end{align*}
	
	It only remains to verify the existence of the mapping $g$. To do this we will adopt a similar strategy to that employed in \cite[Theorem~4.1]{McM}. By Hall's Marriage Theorem it suffices to verify that $\left|A\right|\leq\left|R(A)\right|$ for any set $A\subseteq X$.
	
	Let $A\subseteq X$, $T_{1},\ldots,T_{p}$ be an enumeration of $\set{T(x)\colon x\in A}$ and $E:=\bigcup_{j\in[p]}f(T_{j})$. Then
	\begin{equation*}
		\set{y\in Y\colon d_{\infty}(y,\cl{E})\leq \frac{1}{n}}=\set{y\in Y\colon d_{\infty}(y,E)\leq \frac{1}{n}}=R(A).
	\end{equation*}
	Moreover, $\cl{E}=\bigcup_{j\in[p]}\cl{f(T_{j})}$ is $\nu$-measurable because it is closed and $\nu$ is Borel. Therefore, by \eqref{upper_bound_nu},
	\begin{equation}\label{eq:ineq1}
		\nu\left(\bigcup_{j\in [p]}\cl{f(T_{j})}\right)\leq \frac{1}{n^{d}}\cdot \left|R(A)\right|.
	\end{equation}
	On the other hand, using $f_{\sharp}\mu=\nu$ and \eqref{lower_bound_mu} we may derive
	\begin{multline}
		\nu\left(\bigcup_{j\in[p]}\cl{f(T_{j})}\right)\geq\sum_{j\in [p]}\mu(T_{j})\geq\frac{1}{n^{d}}\sum_{j\in[p]}\abs{X\cap T_{j}}\geq \frac{1}{n^{d}}\cdot\left|A\right|. \label{eq:ineq2}
	\end{multline}
	Combining \eqref{eq:ineq1} and \eqref{eq:ineq2} we get
	\begin{equation*}
		\left|A\right|\leq\left|R(A)\right|,
	\end{equation*}
	as required.
\end{proof}

\section{Well-distributed sets.}\label{sec:well_dist}
In this section we derive an upper bound on the best Lipschitz constant $F_{n}(S)$ for sets $S\in \binom{[cn]^{d}}{n^{d}}$ which are `well-distributed' in the sense that the points are quite evenly spread, relative to the grid partition coming from $\mc{T}_{m}$. 
\begin{lemma}\label{lemma:well_dist_set_bound}
	Let $m,n,l\in \N$ with $m=2^{l}\leq n$, $0<\theta<a<b<\frac{1}{\theta}$, $c>1$ and $S\subseteq \Z^{d}\cap [0,c n]^{d}$ be a finite set with $\abs{S}=n^{d}$ and 
	\begin{equation*}
		\frac{an^{d}}{m^{d}}\leq \abs{S\cap (c n\cdot T)}\leq \frac{bn^{d}}{m^{d}}
	\end{equation*}
	for all $T\in \mc{T}_{m}$, where $\mc{T}_{m}$ is defined by \eqref{def:Tk}. Then there exists a bijection $g\colon S\to [n]^{d}$ and constants $\Lambda:=\Lambda(d)$, $\Delta:=\Delta(d,\theta)$ with 
	\begin{equation*}
		\lip(g)\leq \Lambda \exp\left(\log n-l(1-\Delta(b-a))\right).
	\end{equation*}
\end{lemma} 
Let us begin working towards a proof of Lemma~\ref{lemma:well_dist_set_bound}. The bound will be established by applying Lemma~\ref{lemma:cts_to_discrete} in the case that $\nu$ is the Lebesgue measure on $[0,1]^{d}$ and $\mu$ has the form $\mu=\rho\leb$, where $\rho$ is of the form considered in the next lemma.
\begin{lemma}\label{lemma:smoothing_tiles}
	Let $l\in\N$, $\theta\in (0,1)$, $\mc{T}_{k}$ be defined by \eqref{def:Tk} for each $k\in\N$, and $\rho\colon [0,1]^{d}\to(0,\infty)$ be a function such that $\rho|_{T}$ is constant for each $T\in \mc{T}_{2^{l}}$, $\int_{[0,1]^{d}}\rho\,d\leb=1$ and $\theta\leq \min\rho\leq \max\rho\leq \frac{1}{\theta}$. Then there exists a Lipschitz homeomorphism $f\colon [0,1]^{d}\to\R^{d}$ and a constant $\Delta=\Delta(d,\theta)>0$ such that
	\begin{equation*}
		f_{\sharp}\rho\leb=\leb|_{[0,1]^{d}},
	\end{equation*}
	and
	\begin{equation*}
		\lip(f)\leq (1+\Delta(\max \rho-\min \rho))^{l}.
	\end{equation*}
\end{lemma}	
The proof of Lemma~\ref{lemma:smoothing_tiles} is due to Rivier\'e and Ye~\cite{RY}. However, there the argument is used to prove a more general statement and Lemma~\ref{lemma:smoothing_tiles} is not stated or proved explicitly. The proof is based on the following lemma. 
\begin{lemma}[{\cite[Lemma~1]{RY}}]\label{lemma:RY}
	Let $D=[0,1]^{d}$, $A=[0,1]^{d-1}\times\left[0,\frac{1}{2}\right]$ and $B=[0,1]^{d-1}\times\left[\frac{1}{2},1\right]$. Let $\alpha,\beta\geq 0$ be such that $\alpha+\beta=1$ and let $\eta>0$ be such that $\eta\leq \alpha\leq 1-\eta$. Then there exists a Lipschitz homeomorphism $\Phi\colon [0,1]^{d}\to [0,1]^{d}$ and a constant $\Delta=\Delta(d,\eta)$ such that
	\begin{enumerate}[(i)]
		\item\label{RY1} $\Phi|_{\partial [0,1]^{d}}=\id_{\partial[0,1]^{d}}$,
		\item\label{RY2} $\Jac(\Phi)(x)=\begin{cases}
			2\alpha & \text{if }x\in A,\\
			2\beta & \text{if }x\in B,
		\end{cases}$ for a.e. $x\in [0,1]^{d}$,
		\item\label{RY3} $\lip(\Phi-\id)\leq \Delta \abs{1-2\alpha}$.
	\end{enumerate}
\end{lemma}
Since we only require the argument of Rivier\'e and Ye~\cite{RY} for a particular special case, the following restricted version of the argument is more convenient for the reader.
\begin{proof}[Proof of Lemma~\ref{lemma:smoothing_tiles}]
	For each $i\in\N\cup\set{0}$ and $k=(k_{1},\ldots,k_{d})\in (\set{0}\cup[2^{i}-1])^{d}$ we let 
	\begin{equation*}
		C(k,i):=\prod_{1\leq j\leq d}\left[\frac{k_{j}}{2^{i}},\frac{k_{j}+1}{2^{i}}\right].
	\end{equation*}
	Observe that the sets $C(k,i)$ for $k\in(\set{0}\cup[2^{i}-1])^{d}$ are, up until sets of Lebesgue measure zero, the same as the sets in $\mc{T}_{2^i}$. In particular we have that the restriction of $\rho$ to each $C(k,l)$ is a.e. constant.
	
	For each $i\in\N\cup\set{0}$, we define a homeomorphism $\Phi_{i}\colon [0,1]^{d}\to [0,1]^{d}$ by prescribing it on each cube $C(k,i)$, $k=(k_{1},\ldots,k_{d})\in (\set{0}\cup[2^{i}-1])^{d}$. 
	
	Fix $i\in\N\cup\set{0}$ and $k=(k_{1},\ldots,k_{d})\in (\set{0}\cup[2^{i}-1])^{d}$. For each $p\in[d]$ and $\varepsilon=(\varepsilon_{1},\ldots,\varepsilon_{d})\in\set{0,1}^{d}$ let
	\begin{align*}
		A_{i}^{p}(\varepsilon)&=\prod_{j=1}^{p-1}\left[\frac{k_{j}}{2^{i}},\frac{k_{j}+1}{2^{i}}\right]\times \left[\frac{k_{p}}{2^{i}},\frac{2k_{p}+1}{2^{i+1}}\right]\times\prod_{j=p+1}^{d}\left[\frac{k_{j}}{2^{i}}+\frac{\varepsilon_{j}}{2^{i+1}},\frac{k_{j}}{2^{i}}+\frac{\varepsilon_{j}+1}{2^{i+1}}\right],\\
		B_{i}^{p}(\varepsilon)&=\prod_{j=1}^{p-1}\left[\frac{k_{j}}{2^{i}},\frac{k_{j}+1}{2^{i}}\right]\times \left[\frac{2k_{p}+1}{2^{i+1}},\frac{k_{p}+1}{2^{i}}\right]\times\prod_{j=p+1}^{d}\left[\frac{k_{j}}{2^{i}}+\frac{\varepsilon_{j}}{2^{i+1}},\frac{k_{j}}{2^{i}}+\frac{\varepsilon_{j}+1}{2^{i+1}}\right],\\
		&\qquad\alpha_{i}^{p}(\varepsilon)=\frac{\int_{A_{i}^{p}(\varepsilon)}\rho \,d\leb}{\int_{A_{i}^{p}(\varepsilon)\cup B_{i}^{p}(\varepsilon)}\rho\,d\leb},\qquad\qquad \beta_{i}^{p}(\varepsilon)=\frac{\int_{B_{i}^{p}(\varepsilon)}\rho \,d\leb}{\int_{A_{i}^{p}(\varepsilon)\cup B_{i}^{p}(\varepsilon)}\rho\,d\leb}. 
	\end{align*}
	Note that, for each fixed $p$, the sets $A_{i}^{p}(\varepsilon)$, $B_{i}^{p}(\varepsilon)$ indexed by $\varepsilon=(\varepsilon_{1},\ldots,\varepsilon_{d})\in \set{0,1}^{d}$ determine a partition of $C(k,i)$. More precisely, after ignoring repetitions, these sets have pairwise disjoint interiors and their union is $C(k,i)$.
	
	For each $j\in[d]$ we define a homeomorphism $\Phi_{i}^{j}\colon C(k,i)\to C(k,i)$ as follows: For each $\varepsilon=(\varepsilon_{1},\ldots,\varepsilon_{d})\in \set{0,1}^{d}$ define $\Phi_{i}^{j}|_{A_{i}^{j}(\varepsilon)\cup B_{i}^{j}(\varepsilon)}$ as the homeomorphism given by the conclusion of Lemma~\ref{lemma:RY} applied with $D=A_{i}^{j}(\varepsilon)\cup B_{i}^{j}(\varepsilon)$, $A=A_{i}^{j}(\varepsilon)$, $B=B_{i}^{j}(\varepsilon)$, $\alpha=\alpha_{i}^{j}(\varepsilon)$ and $\beta=\beta_{i}^{j}(\varepsilon)$. Note that here we have to use Lemma~\ref{lemma:RY} in combination with suitable affine transformations. Further, the parameter $\eta$ in Lemma~\ref{lemma:RY} may be taken as $\eta=\frac{\theta^{2}}{1+\theta^{2}}$ and we have
	\begin{equation*}
		\abs{1-2\alpha_{i}^{p}(\varepsilon)}\leq \frac{1}{2\theta}(\max\rho-\min\rho),\qquad p\in[d],\,\varepsilon\in \set{0,1}^{d}.
	\end{equation*}
	Hence, this application of Lemma~\ref{lemma:RY} provides a Lipschitz homeomorphism $$\Phi_{i}^{j}\colon A_{i}^{j}(\varepsilon)\cup B_{i}^{j}(\varepsilon)\to A_{i}^{j}(\varepsilon)\cup B_{i}^{j}(\varepsilon)$$ with properties \eqref{RY1}--\eqref{RY3}, where, for a constant $\Delta=\Delta(d,\theta)>0$, \eqref{RY3} translates to
	\begin{equation*}
		\lip(\Phi_{i}^{j}-\id)\leq \Delta (\max\rho-\min\rho), \qquad j\in [d],
	\end{equation*}
	implying
	\begin{equation}\label{eq:upper_bound_Lip}
	\lip(\Phi_{i}^{j})\leq 1+\Delta (\max\rho-\min\rho), \qquad j\in [d].
	\end{equation}
	Due to property \eqref{RY1}, we can glue all of these homeomorphisms together to obtain a homeomorphism $\Phi_{i}^{j}\colon C(k,i)\to C(k,i)$ preserving \eqref{eq:upper_bound_Lip}. Property \eqref{RY1} then allows us to again glue all of these homeomorphisms constructed on each $C(k,i)$ together to obtain a homeomorphism $\Phi_{i}^{j}\colon [0,1]^{d}\to [0,1]^{d}$ preserving \eqref{eq:upper_bound_Lip}.	
	
	Finally set 
	\begin{align*}
		\Phi_{i}&:=\Phi_{i}^{d}\circ \Phi_{i}^{d-1}\circ \ldots \circ \Phi_{i}^{1}\text{ for $i\in[l-1]\cup\set{0}$,}\\
		f_{q}&:=\Phi_{0}\circ\Phi_{1}\circ \Phi_{2}\circ \ldots \Phi_{q}\text{ for $q\in[l-1]$,} \qquad \text{and }
		f:=f_{l-1}.
	\end{align*}
It can be checked that whenever $C(k',i+1)\subseteq C(k,i)$, we have 
\begin{equation*}
	\Jac(\Phi_{i})|_{C(k',i+1)}\equiv \frac{2^{d}\int_{C(k',i+1)}\rho\,d\leb}{\int_{C(k,i)}\rho\,d\leb}.
\end{equation*}
 Moreover, for each $i$ and $C(k,i)$ we have that $\Phi_{i}|_{C(k,i)}$ is a homeomorphism $C(k,i)\to C(k,i)$. Thus, we may use the chain rule for Jacobians to compute 
\begin{equation*}
	\Jac(f)(x)=\frac{2^{ld}\int_{C(k,l)}\rho\,d\leb}{\int_{[0,1]^{d}}\rho\,d\leb}=\frac{1}{C(k,l)}\int_{C(k,l)}\rho\,d\leb=\rho(x)\quad \text{for all $k$ and a.e. $x\in C(k,l)$},
\end{equation*}
where for the last equality we use that $\rho|_{C(k,l)}$ is a.e. constant for each $C(k,l)$.
Hence $\Jac(f)(x)=\rho(x)$ for a.e. $x\in [0,1]^{d}$ and accordingly $f_{\sharp}\rho\leb=\leb|_{[0,1]^{d}}$. Moreover, we have 
	\begin{multline*}
		\lip(f)\leq \prod_{q=0}^{l-1}\lip(\Phi_{q})\leq \prod_{q=0}^{l-1}\prod_{j=1}^{d}\lip(\Phi_{q}^{j})\\
		\leq (1+\Delta(\max\rho-\min\rho))^{ld}\leq (1+\Delta(\max\rho-\min\rho))^{l},
	\end{multline*}
	where we allow the constant $\Delta=\Delta(d,\theta)$ to increase in the last occurence.
\end{proof}

\begin{proof}[Proof of Lemma~\ref{lemma:well_dist_set_bound}]
	Define $\rho\colon [0,1]^{d}\to (0,\infty)$ by
	\begin{equation*}
		\rho|_{T}\equiv\frac{m^{d}}{n^{d}}\cdot \abs{S\cap (c n\cdot T)}, \quad T\in\mc{T}_{m},\qquad \rho|_{[0,1]^{d}\setminus\bigcup\mc{T}_{m}}\equiv 0.
	\end{equation*}
	Thus, $\rho$ is constant on each $T\in \mc{T}_{m}$ and $a\leq \rho\leq b$. By Lemma~\ref{lemma:smoothing_tiles} there exists a Lipschitz homeomorphism $f\colon [0,1]^{d}\to [0,1]^{d}$ and a constant $\Delta=\Delta(d,\theta)$ such that $f_{\sharp}\rho\leb=\leb|_{[0,1]^{d}}$ and $\lip(f)\leq (1+\Delta(b-a))^{l}$. We may now apply Lemma~\ref{lemma:cts_to_discrete} to $\mu=\rho\leb$, $\nu=\leb|_{[0,1]^{d}}$, $n$, $\mc{T}=\mc{T}_{m}$, $c$, $X=\frac{1}{cn}S$, $Y=\frac{1}{n}[n]^{d}$ and $f$ to get a bijective mapping $\widetilde{g}\colon \frac{1}{c n}\cdot S\to \frac{1}{n}[n]^{d}$ and a constant $\Lambda=\Lambda(d)$ with
	\begin{multline*}
		\lip(\widetilde{g})\leq \Lambda \max\set{\lip(f),1}c\frac{n}{m}\leq\Lambda c2^{\log n-l}(1+\Delta(b-a))^{l}\\
		=\Lambda c \exp\left(\log n-l(1-\log(1+\Delta(b-a)))\right)\leq \Lambda c \exp\left(\log n-l(1-\Delta(b-a))\right).
	\end{multline*}
	Finally, we define $g\colon S\to [n]^{d}$ by 
	\begin{equation*}
		g(x)=n\cdot \widetilde{g}\left(\frac{x}{cn}\right), \qquad x\in S.
	\end{equation*}
\end{proof}

\section{Random sets.}\label{sec:random}
In this section we show that for $C>1$ and large $n$, a random set $S\in \binom{[C^{1/d}n]^{d}}{n^{d}}$ is well-distributed in the sense of Section~\ref{sec:well_dist} with high probability. The statements in this section will be written according to Convention~\ref{conv:randomset}.

Calculating probabilities in the space $\binom{[C^{1/d}n]^{d}}{n^{d}}$ will inevitably lead to expressions involving large binomial coefficients. To estimate these numbers, we will use the following standard lemma which follows easily from Stirling's approximation of the factorial. 

In what follows $H$ denotes the binary entropy function
\begin{equation*}
	H(t)=-t\log t-(1-t)\log(1-t),\qquad t\in [0,1].
\end{equation*}
Later on we will use certain important properties of the binary entropy function $H$, namely that it is strictly convex, differentiable and that its derivative is given by
\begin{equation*}
	H'(t)=-\log\left(\frac{t}{1-t}\right),\qquad t\in (0,1).
\end{equation*}
\begin{lemma}\label{lemma:stirling}
	There is an absolute constant $\Lambda>0$ such that 
	\begin{equation*}
		\begin{cases}
			\Lambda^{-1}\sqrt{\frac{p}{2\pi q(p-q)}}\cdot2^{pH\left(\frac{q}{p}\right)}\leq\binom{p}{q}\leq \Lambda\sqrt{\frac{p}{2\pi q(p-q)}}\cdot 2^{pH\left(\frac{q}{p}\right)} & \text{if }q\in[p-1]\setminus\set{0},\\
			\Lambda^{-1}2^{pH\left(\frac{q}{p}\right)}\leq \binom{p}{q}\leq \Lambda 2^{pH\left(\frac{q}{p}\right)} & \text{if }q\in\set{0,p}.
		\end{cases},\qquad p\in\N.
	\end{equation*}
\end{lemma}
Note that the inequalities of Lemma~\ref{lemma:stirling} for the case $q\in\set{0,p}$ are trivial, because $\binom{p}{0}=\binom{p}{p}=1$. We write them here because we wish to treat the case $q\in\set{0,p}$ together with the case $q\in [p-1]$ later on.
\begin{proof}[Proof of Lemma~\ref{lemma:stirling}]
	By Stirling's Approximation of $n!$ (see for example~\cite{romik2000stirling}), the quantities
	\begin{equation*}
		\alpha:=\inf_{n\in\N}\frac{n!}{\sqrt{2\pi n}\left(\frac{n}{e}\right)^{n}}>0,
	\end{equation*}
	and 
	\begin{equation*}
		\beta:=\sup_{n\in\N}\frac{n!}{\sqrt{2\pi n}\left(\frac{n}{e}\right)^{n}}<\infty,
	\end{equation*}
	are absolute constants. Let $p\in\N$ and $q\in [p-1]\setminus\set{0}$. Then,
	\begin{equation*}
		\binom{p}{q}=\frac{p!}{q!(p-q)!}\leq \frac{\beta}{\alpha^{2}}\cdot \sqrt{\frac{p}{2\pi q(p-q)}}\cdot \frac{p^{p}}{q^{q}(p-q)^{p-q}}=\frac{\beta}{\alpha^{2}}\cdot \sqrt{\frac{p}{2\pi q(p-q)}}\cdot 2^{pH\left(\frac{q}{p}\right)},
	\end{equation*}
	and similarly
	\begin{equation*}
		\binom{p}{q}\geq \frac{\alpha}{\beta^{2}}\cdot \sqrt{\frac{p}{2\pi q(p-q)}}\cdot 2^{pH\left(\frac{q}{p}\right)}. 
	\end{equation*}
	Now let $p\in\N$ and $q\in\set{0,p}$. Then, $H\left(\frac{q}{p}\right)=0$ and
	\begin{equation*}
		\binom{p}{q}=1= 2^{pH\left(\frac{q}{p}\right)}.
	\end{equation*}
	Therefore, we may take $\Lambda=\frac{\max\set{\beta,\beta^{2}}}{\min\set{\alpha,\alpha^{2}}}$.
\end{proof}

\begin{lemma}\label{lemma:vojtech_concave_inequality}
	Let $I\subseteq \R$ be an open interval, $f\colon I\to (0,\infty)$ be a differentiable, concave and strictly increasing function and let $s,t\in I$ with $s<t$ and let $\lambda\in(0,1)$. Then
	\begin{equation*}
		f((1-\lambda)s+\lambda t)\leq \frac{f'(s)}{f'(t)}\cdot\left( (1-\lambda)f(s)+\lambda f(t))\right).
	\end{equation*}
\end{lemma}
\begin{proof}
	If the inequality holds for the function $g(u):=f(u)-f(s)$ in place of $f$ then it also holds for $f$. This is readily verified using the concavity and positivity of $f$. Thus, we may assume that $f(s)=0$. This allows us to write
	\begin{equation*}
		\frac{f((1-\lambda)s+\lambda t)}{(1-\lambda)f(s)+\lambda f(t)}=\frac{\int_{s}^{(1-\lambda)s+\lambda t}f'(u)\,du}{\lambda \int_{s}^{t}f'(u)\,du}\leq \frac{f'(s)}{f'(t)}.
	\end{equation*}
\end{proof}

\begin{lemma}\label{lemma:bonnet}
	Let $\delta\in [0,\frac{1}{2})$, $N\in\N$, $M>1$, $\frac{1}{2}<a<1-\delta$, $1+2\delta<b<2$, $X$ be a finite set with $\abs{X}>bN$ and $Y\subseteq X$ be a set with
	\begin{equation}\label{eq:Y_cardinality}
		\frac{\left(1-\delta\right)\abs{X}}{M}\leq \abs{Y}\leq \frac{\left(1+\delta\right)\abs{X}}{M}.
	\end{equation}
	Then, there is an absolute constant $\Gamma>0$ such that a random set $S\in \binom{X}{N}$ satisfies
	\begin{multline}\label{eq:bonnet_prob<a}
		\Prob\left[\abs{S\cap Y}\leq \frac{aN}{M}\right]\\
		\leq \Gamma\cdot \sqrt{\frac{\abs{X}-N}{\abs{X}\left(1-\frac{2}{M}\right)-N}}\cdot\frac{N^{3/2}}{M}\cdot \exp\left(-\frac{(1-(a+\delta))^{2}N(\abs{X}-N)}{\Gamma M(\abs{X}-(a+\delta)N)}\right)
	\end{multline}
	and
	\begin{multline}\label{eq:bonnet_prob>b}
		\Prob\left[\abs{S\cap Y}\geq \frac{bN}{M}\right]\\
		\leq \Gamma\cdot \sqrt{\frac{\abs{X}-N}{\abs{X}\left(1-\frac{2}{M}\right)-N}}\cdot N^{3/2}\cdot\exp\left(-\frac{(b-2\delta-1)^{2}N(\abs{X}-(b-2\delta)N)}{\Gamma M(\abs{X}-N)}\right).
	\end{multline}
\end{lemma}
\begin{proof}
	The probabilities considered in \eqref{eq:bonnet_prob<a} and \eqref{eq:bonnet_prob>b} are bounded above by
	\begin{equation*}
		\binom{\abs{X}}{N}^{-1}\sum_{k}\binom{\abs{Y}}{k}\binom{\abs{X}-\abs{Y}}{N-k},
	\end{equation*}
	where the sum is taken over $0\leq k\leq\frac{aN}{M}$ for \eqref{eq:bonnet_prob<a} and over $\frac{bN}{M}\leq k\leq \min\set{N,\abs{Y}}$ for \eqref{eq:bonnet_prob>b}. Our first aim is to establish an upper bound for the quantity
	\begin{equation}\label{eq:bin_coeff_triple}
		\binom{\abs{X}}{N}^{-1}\binom{\abs{Y}}{k}\binom{\abs{X}-\abs{Y}}{N-k}
	\end{equation}
	for $0\leq k\leq \min\set{N,\abs{Y}}$. 
	
	Fix $0\leq k\leq \min\set{N,\abs{Y}}$ and define
	\begin{align*}
		V_{k}&:=\begin{cases}
			\sqrt{\frac{\abs{Y}}{k(\abs{Y}-k)}} & \text{ if }k\notin \set{0,\abs{Y}},\\
			1 & \text{ if }k\in\set{0,\abs{Y}},
		\end{cases}\\
		W_{k}&:=\begin{cases}
			\sqrt{\frac{\abs{X}-\abs{Y}}{(N-k)(\abs{X}-\abs{Y}-N+k)}} & \text{ if }N-k\notin \set{0,\abs{X}-\abs{Y}},\\
			1 & \text{ if }N-k\in\set{0,\abs{X}-\abs{Y}}.
		\end{cases}
	\end{align*}
	Then, we may use Lemma~\ref{lemma:stirling} to bound the product in \eqref{eq:bin_coeff_triple} above by
	\begin{equation}\label{eq:app_stirling}
		\Lambda \cdot\left(\sqrt{\frac{\abs{X}}{N(\abs{X}-N)}}\right)^{-1}V_{k}W_{k}\cdot\exp\left(-\abs{Y}\gamma_{k}\right),
	\end{equation}
	where $\Lambda>0$ is an absolute constant and
	\begin{equation}\label{eq:formula_gamma_k}
		\gamma_{k}:=\frac{\abs{X}}{\abs{Y}}H\left(\frac{N}{\abs{X}}\right)-H\left(\frac{k}{\abs{Y}}\right)-\left(\frac{\abs{X}}{\abs{Y}}-1\right)H\left(\frac{N-k}{\abs{X}-\abs{Y}}\right).
	\end{equation}
	The product $\left(\sqrt{\frac{\abs{X}}{N(\abs{X}-N)}}\right)^{-1}V_{k}W_{k}$ in \eqref{eq:app_stirling} may be bounded above by
	\begin{multline*}
		\sqrt{\frac{\abs{Y}}{\abs{Y}-1}}\cdot\sqrt{\frac{N}{N-(N-1)}}\cdot\sqrt{\frac{\abs{X}-N}{\abs{X}-N-\abs{Y}}}
		\leq \sqrt{2}\sqrt{N}\sqrt{\frac{\abs{X}-N}{\abs{X}-N-\abs{Y}}}.
	\end{multline*}
	Therefore, we obtain an absolute constant $\Gamma>0$ such that
	\begin{equation}\label{eq:new_upper_bound}
		\binom{\abs{X}}{N}^{-1}\binom{\abs{Y}}{k}\binom{\abs{X}-\abs{Y}}{N-k}\leq \Gamma \sqrt{\frac{N(\abs{X}-N)}{\abs{X}-N-\abs{Y}}}\cdot \exp\left(-\abs{Y}\gamma_{k}\right).
	\end{equation} 
	
	Our task is now to establish a lower bound on $\gamma_{k}$. To this end, we rewrite the formula~\eqref{eq:formula_gamma_k} for $\gamma_{k}$ as
	\begin{align*}
		\gamma_{k}&=H\left(\frac{N}{\abs{X}}\right)-H\left(\frac{k}{\abs{Y}}\right)-\left(\frac{\abs{X}}{\abs{Y}}-1\right)\left(H\left(\frac{N-k}{\abs{X}-\abs{Y}}\right)-H\left(\frac{N}{\abs{X}}\right)\right)\\
		&=H\left(\frac{N}{\abs{X}}\right)-H\left(\frac{k}{\abs{Y}}\right)-\left(\frac{N}{\abs{X}}-\frac{k}{\abs{Y}}\right)H'(\xi_{k})
	\end{align*}
	for some $\xi_{k}$ lying in the interval with endpoints $\frac{N}{\abs{X}}$ and $\frac{N-k}{\abs{X}-\abs{Y}}$. From the bounds on $\abs{Y}$, $a$ and $b$ given by the hypothesis of the lemma, we have
	\begin{align}
		\frac{k}{\abs{Y}}<\frac{N}{\abs{X}}\leq \xi_{k}\leq\frac{N-k}{\abs{X}-\abs{Y}} \qquad &\text{if $0\leq k\leq \frac{aN}{M}$,} \label{eq:xi_k_1}\\
		\frac{N-k}{\abs{X}-\abs{Y}}\leq \xi_{k}\leq\frac{N}{\abs{X}}<\frac{k}{\abs{Y}}\qquad &\text{if $\frac{bN}{M}\leq k\leq \min\set{N,\abs{Y}}$.} \label{eq:xi_k_2}
	\end{align}
	Assume first, that $0\leq k\leq \frac{aN}{M}$.
	Then \eqref{eq:xi_k_1}, together with that fact that $H$ is strictly concave, allows us to write
	\begin{multline*}
		\gamma_{k}=\int_{\frac{k}{\abs{Y}}}^{\frac{N}{\abs{X}}}H'(t)-H'(\xi_{k})\,dt\geq \int_{\frac{(a+\delta)N}{\abs{X}}}^{\frac{N}{\abs{X}}}H'(t)-H'\left(\frac{N}{\abs{X}}\right)\,dt\\
		\geq H\left(\frac{N}{\abs{X}}\right)-H\left(\frac{((a+\delta)N}{\abs{X}}\right)-\frac{(1-(a+\delta))N}{\abs{X}}H'\left(\frac{N}{\abs{X}}\right)\\
		=-\frac{N}{\abs{X}}\log \left(\frac{N}{\abs{X}}\right)-\left(1-\frac{N}{\abs{X}}\right)\log \left(1-\frac{N}{\abs{X}}\right)+\frac{(a+\delta)N}{\abs{X}}\log\left(\frac{(a+\delta)N}{\abs{X}}\right)\\+\left(1-\frac{(a+\delta)N}{\abs{X}}\right)\log\left(1-\frac{(a+\delta)N}{\abs{X}}\right)+\frac{(1-(a+\delta))N}{\abs{X}}\log\left(\frac{\frac{N}{\abs{X}}}{1-\frac{N}{\abs{X}}}\right)\\
		=\left(1-\frac{(a+\delta)N}{\abs{X}}\right)\log\left(\frac{\abs{X}-(a+\delta)N}{\abs{X}-N}\right)+\frac{(a+\delta)N}{\abs{X}} \log(a+\delta).
	\end{multline*}
	Finally, we apply Lemma~\ref{lemma:vojtech_concave_inequality} to $I=(0,\infty)$, $f=\log$, $s=a$ and $t=\frac{\abs{X}-(a+\delta)N}{\abs{X}-N}$ in order to bound the latter expression below by
	\begin{equation}\label{eq:final_gamma_k_bound1}
		\frac{\log'\left(\frac{\abs{X}-(a+\delta)N}{\abs{X}-N}\right)}{\log'(a+\delta)}\log\left(1+\frac{(1-(a+\delta))^{2}N}{\abs{X}-N}\right)\geq \frac{(1-(a+\delta))^{2}N(\abs{X}-N)}{2\abs{X}(\abs{X}-(a+\delta)N)},
	\end{equation}
	where the latter inequality is derived by applying the inequality $\log(1+x)\geq \frac{x}{1+x}$. Similarly, if $\frac{bN}{M}\leq k\leq \min\set{N,\abs{Y}}$, we use \eqref{eq:xi_k_2} and the strict concavity of $H$ to derive
	\begin{align}
		\gamma_{k}=\int_{\frac{N}{\abs{X}}}^{\frac{k}{\abs{Y}}}H'(\xi_{k})-H'(t)\,dt&\geq \int_{\frac{N}{\abs{X}}}^{\frac{(b-2\delta)N}{\abs{X}}}H'\left(\frac{N}{\abs{X}}\right)-H'(t)\,dt\nonumber \\
		&\geq \frac{(b-2\delta-1)N}{\abs{X}}H'\left(\frac{N}{\abs{X}}\right)-\left(H\left(\frac{(b-2\delta)N}{\abs{X}}\right)-H\left(\frac{N}{\abs{X}}\right)\right)\nonumber\\
		&\geq\frac{-(b-2\delta-1)N}{\abs{X}}\log\left(\frac{\frac{N}{\abs{X}}}{1-\frac{N}{\abs{X}}}\right)+\frac{(b-2\delta)N}{\abs{X}}\log \left(\frac{(b-2\delta)N}{\abs{X}}\right)\nonumber \\
		&\qquad+\left(1-\frac{(b-2\delta)N}{\abs{X}}\right)\log \left(1-\frac{(b-2\delta)N}{\abs{X}}\right)\nonumber\\
		&\qquad-\frac{N}{\abs{X}}\log\left(\frac{N}{\abs{X}}\right)-\left(1-\frac{N}{\abs{X}}\right)\log\left(1-\frac{N}{\abs{X}}\right)\nonumber\\
		&=\frac{(b-2\delta)N}{\abs{X}}\log(b-2\delta)+\left(1-\frac{(b-2\delta)N}{\abs{X}}\right)\log\left(\frac{\abs{X}-(b-2\delta)N}{\abs{X}-N}\right)\nonumber\\
		&\geq \frac{\log'(b-2\delta)}{\log'\left(\frac{\abs{X}-(b-2\delta)N}{\abs{X}-N}\right)}\cdot \log\left(1+\frac{(b-2\delta-1)^{2}N}{\abs{X}-N}\right)\nonumber\\
		&\geq \frac{(b-2\delta-1)^{2}N(\abs{X}-(b-2\delta)N)}{2\abs{X}(\abs{X}-N)}. \label{eq:final_gamma_k_bound2}
	\end{align}	
	Finally, we substitute the lower bounds \eqref{eq:final_gamma_k_bound1} and \eqref{eq:final_gamma_k_bound2} for $\gamma_{k}$ into \eqref{eq:new_upper_bound}, to acquire upper bounds on the product in \eqref{eq:bin_coeff_triple} in the cases $0\leq k\leq \frac{aN}{M}$ and $\frac{bN}{M}\leq k\leq \min\set{N,\abs{Y}}$ respectively. Moreover, in both cases these upper bounds are independent of $k$. Thus, by summing the relevant upper bounds over $0\leq k\leq \frac{aN}{M}$ and $\frac{bN}{M}\leq k\leq \min\set{N,\abs{Y}}$ respectively and additionally applying the bounds on $\abs{Y}$ from \eqref{eq:Y_cardinality}, we establish \eqref{eq:bonnet_prob<a} and \eqref{eq:bonnet_prob>b}. In case of possible future relevance, we point out that the factor $N^{3/2}$ in \eqref{eq:bonnet_prob>b} may be replaced by $N^{1/2}\cdot\min\set{N,\abs{Y}}$. This comes from keeping the term $\min\set{N,\abs{Y}}$ when summing over $\frac{bN}{M}\leq k\leq \min\set{N,\abs{Y}}$, rather than bounding it above by $N$, as we do, for simplicity, to get \eqref{eq:bonnet_prob>b}.
\end{proof}
\begin{lemma}\label{lemma:log_threshold_prob}
	Let $d,n,m\in\N$ and $C,q\in\R$ with
	\begin{equation}\label{eq:cond_m_C_q}
		\frac{n}{2(\log n)^{q}}\leq m\leq \frac{2n}{(\log n)^{q}},\qquad C\geq 1+\frac{2^{d+7}}{\log n},\qquad \begin{cases} q\geq 1 & \text{ if }\frac{C^{1/d}n}{m}\notin \Z \\
			q> 0 & \text{ if }\frac{C^{1/d}n}{m}\in \Z 
		\end{cases}.
	\end{equation}
	Let $\mc{T}_{m}$ be defined by \eqref{def:Tk}. Then there exists a constant $\Gamma=\Gamma(d)>0$ such that a random set $S\in \binom{[C^{1/d}n]^{d}}{n^{d}}$ satisfies
	\begin{multline}\label{eq:lower_log}
		\Prob\left[\exists T\in\mc{T}_{m}\text{ s.t. }\abs{S\cap \left(C^{1/d}n\cdot T\right)}
		\leq \frac{\left(1-\frac{\Gamma}{\log n}\right)n^{d}}{m^{d}}\right]\\
		\leq \Gamma n^{ \Gamma}\exp\left(-\frac{(\log n)^{qd-2}}{\Gamma}\right)
	\end{multline}
	and
	\begin{multline}\label{eq:upper_log}
		\Prob\left[\exists T\in\mc{T}_{m}\text{ s.t. }\abs{S\cap \left(C^{1/d}n\cdot T\right)}\geq \frac{\left(1+\frac{\Gamma}{\log n}\right)n^{d}}{m^{d}}\right]\\
		\leq \Gamma  n^{\Gamma}\exp\left(-\frac{(\log n)^{qd-2}}{\Gamma}\right).
	\end{multline} 
\end{lemma}
\begin{proof}
	In the present proof, $\Gamma$ will always denote a (large) constant which may depend only on $d$ and whose value is allowed to increase in each occurence. So, to give an example of the use of this convention, we would write the inequality $\Gamma n^{d}+dn\leq \Gamma n^{d}$ for $n\in\N$ instead of writing $\Gamma n^{d}+dn\leq (\Gamma+d)n^{d}$. Moreover, we point out that it suffices to verify the conclusions \eqref{eq:lower_log} and \eqref{eq:upper_log} of the lemma with an additional assumption that $n$ is larger than some threshold depending only on $d$. The finitely many remaining $n\in\N$ can then be treated by adjusting the constant $\Gamma=\Gamma(d)$ if necessary. Therefore, in the present proof, every inequality involving $n$ should be read with an additional condition that $n$ is sufficiently large, where the sufficiently large condition depends only on $d$.  
	
	Fix $T\in \mc{T}_{m}$ and let $X:=[C^{1/d}n]^{d}$ and $Y:=\left(C^{1/d}n\cdot T\right)\cap X$. Then,
	\begin{align*}
		Cn^{d}\left(1-\frac{2^{d}}{n}\right)\leq \abs{X}&\leq Cn^{d},\quad \text{ and}\\
		\frac{Cn^{d}}{m^{d}}\left(1-\frac{2^{d} m}{n}\right)\leq \abs{Y}&\leq \frac{Cn^{d}}{m^{d}}\left(1+\frac{2^{d} m}{n}\right).
	\end{align*}
	These inequalities imply
	\begin{equation*}
		\frac{\left(1-\frac{2^{d}m}{n}\right)\abs{X}}{m^{d}}\leq \abs{Y}\leq \frac{\left(1+\frac{2^{d+2}m}{n}\right)\abs{X}}{m^{d}}.
	\end{equation*}
	In the special case that $\frac{C^{1/d}n}{m}\in \Z$, we note that $\abs{Y}=\frac{Cn^{d}}{m^{d}}=\frac{\abs{X}}{m^{d}}$.
	
	Set $N=n^{d}$, $M:=m^{d}$ and $\delta:=\begin{cases}
		\frac{2^{d+2}m}{n} & \text{if } \frac{C^{1/d}n}{m}\notin \Z\\
		0 & \text{if } \frac{C^{1/d}n}{m}\in \Z
	\end{cases}$, so that \eqref{eq:Y_cardinality} is satisfied and
	\begin{equation}\label{eq:delta}
		\frac{2^{d+1}}{(\log n)^{q}}\leq \delta\leq \frac{2^{d+3}}{(\log n)^{q}}, \qquad\text{ if }\frac{C^{1/d}n}{m}\notin \Z.
	\end{equation}
	We apply Lemma~\ref{lemma:bonnet} to $\delta$, $N$, $M$, $a:=1-\frac{2^{d+5}}{\log n}$, $b=1+\frac{2^{d+5}}{\log n}$, $X$ and $Y$. After applying the bounds or substituting the values for the parameters, the probability inequalities \eqref{eq:bonnet_prob<a} and \eqref{eq:bonnet_prob>b} given by Lemma~\ref{lemma:bonnet} become
	\begin{equation}\label{eq:plug_evryth_in_1}
		\Prob\left[\abs{S\cap \left(C^{1/d}n\cdot T\right)}\leq \frac{\left(1-\frac{2^{d+5}}{\log n}\right)n^{d}}{m^{d}}\right]
		\leq \Gamma n^{d/2}(\log n)^{qd}\exp\left(-\frac{(\log n)^{qd-2}}{\Gamma}\right),
	\end{equation}
	and
	\begin{equation}\label{eq:plug_evryth_in_2}
		\Prob\left[\abs{S\cap \left(C^{1/d}n\cdot T\right)}\geq \frac{\left(1+\frac{2^{d+5}}{\log n}\right)n^{d}}{m^{d}}\right]
		\leq\Gamma n^{3d/2}\exp\left(-\frac{(\log n)^{qd-2}}{\Gamma}\right).
	\end{equation}
	To aid in the verification of \eqref{eq:plug_evryth_in_1} and \eqref{eq:plug_evryth_in_2} we list the following utilised bounds on terms from \eqref{eq:bonnet_prob<a} and \eqref{eq:bonnet_prob>b}:
	\begin{align}
		\frac{\abs{X}-N}{\abs{X}\left(1-\frac{2}{M}\right)-N}&\leq 2,\label{eq:pesky1}\\
		\frac{\abs{X}-N}{\abs{X}-(a+\delta)N}&\geq \frac{1}{2},\label{eq:pesky2}\\
		\frac{\abs{X}-(b-2\delta)N}{\abs{X}-N}&\geq \frac{1}{2},\label{eq:pesky3}
	\end{align}
	We presently explain how to verify each of the bounds \eqref{eq:pesky1}--\eqref{eq:pesky3}: For \eqref{eq:pesky1}, first note that
	\begin{equation*}
		\frac{\abs{X}-N}{\abs{X}(1-\frac{2}{M})-N}\leq \frac{(C-1)}{C\left(1-\frac{2^{d}}{n}\right)\left(1-\frac{2}{M}\right)-1}\leq \frac{C-1}{(C-1)-C\left(\frac{2^{d}}{n}+\frac{2}{M}\right)},
	\end{equation*}
	and then observe that
	\begin{equation*}
		\\
		\frac{2^{d}}{n}+\frac{2}{M}<\frac{1}{\log n}<\frac{C-1}{2C}.
	\end{equation*}
	For \eqref{eq:pesky2}, observe that
	\begin{equation*}
		\frac{\abs{X}-N}{\abs{X}-(a+\delta)N}\geq \frac{C\left(1-\frac{2^{d}}{n}\right)-1}{C-(a+\delta)},
	\end{equation*}
	and the inequality $\frac{C\left(1-\frac{2^{d}}{n}\right)-1}{C-(a+\delta)}\geq\frac{1}{2}$ is equivalent to 
	\begin{equation*}
		C\geq\frac{1-\frac{(a+\delta)}{2}}{\frac{1}{2}-\frac{2^{d}}{n}}= 1+\frac{\frac{2^{d+4}}{\log n}-\frac{\delta}{2}+\frac{2^{d}}{n}}{\frac{1}{2}-\frac{2^{d}}{n}},
	\end{equation*}
	which evidently holds, in light of \eqref{eq:cond_m_C_q} and \eqref{eq:delta}. The verification of \eqref{eq:pesky3} can be done similarly to that of \eqref{eq:pesky2}.
	
	Having established \eqref{eq:plug_evryth_in_1} and \eqref{eq:plug_evryth_in_2}, we obtain \eqref{eq:lower_log} and \eqref{eq:upper_log} by summing \eqref{eq:plug_evryth_in_1} and \eqref{eq:plug_evryth_in_2} over $T\in\mc{T}_{m}$ and applying the bound $\abs{\mc{T}_{m}}=m^{d}\leq \frac{2^{d}n^{d}}{(\log n)^{qd}}$.
\end{proof}

\section{Proof of Main Result.}\label{sec:proof}
To finish this note, we give a proof of Theorem~\ref{thm:aas}. For the reader's convenience, we repeat the statement here:
\aas*
\begin{proof}
	In this proof we will adopt the same convention with the constant $\Gamma$ as used in the proof of Lemma~\ref{lemma:log_threshold_prob}, see the start of the proof of Lemma~\ref{lemma:log_threshold_prob} for an explanation. 	
	
	Set $l_{n}:=\floor{\log n -q\log\log n}$ and $m_{n}:=2^{l_{n}}$ for all $n\in\N$ starting at a certain threshold so that all expressions make sense. For the finitely many remaining $n$ we define $m_{n}$ in the same way, but set $l_{n}=1$.
	
	We set $C_{n}:=c_{n}^{d}$. The conditions of Lemma~\ref{lemma:log_threshold_prob} are satisfied for $d$, $n$, $m=m_{n}$, $C=C_{n}$ and $q$. Applying Lemma~\ref{lemma:log_threshold_prob}, we deduce that there is a constant $\Gamma=\Gamma(d)>0$ such that
	\begin{multline}\label{eq:well_dist_high_prob}
		\Prob\left[\frac{\left(1-\frac{\Gamma}{\log n}\right)n^{d}}{m_{n}^{d}}\leq \abs{S\cap \left(C_{n}^{1/d}n\cdot T\right)}\leq \frac{\left(1+\frac{\Gamma}{\log n}\right)n^{d}}{m_{n}^{d}}\quad \text{for all }T\in\mc{T}_{m}\right]\\
		\geq 1-\Gamma n^{\Gamma}\exp\left(-\frac{(\log n)^{qd-2}}{\Gamma}\right),
	\end{multline}
	for all $n\in\N$ and a random set $S\in \Omega_{n}$; see Convention~\ref{conv:randomset}. Let $\Lambda=\Lambda(d)>0$ and $\Delta=\Delta(d,\frac{1}{2})$ be the constants given by the conclusion of Lemma~\ref{lemma:well_dist_set_bound}. Then, combining \eqref{eq:well_dist_high_prob} and Lemma~\ref{lemma:well_dist_set_bound}, we conclude that
	\begin{equation*}
		\Prob\left[F_{n}> \Lambda\exp\left(\log n-l_{n}\left(1-\frac{2\Delta\Gamma}{\log n}\right)\right)\right]\leq \Gamma n^{\Gamma}\exp\left(-\frac{(\log n)^{qd-2}}{\Gamma}\right),
	\end{equation*}
	for all $n\in \N$. To finish the proof, it only remains to observe
	\begin{multline*}
		\exp\left(\log n-l_{n}\left(1-\frac{2\Delta\Gamma}{\log n}\right)\right)\leq \exp\left(\log n-(\log n-q\log\log n-1)\cdot\left(1-\frac{2\Delta\Gamma}{\log n}\right)\right)\\
		=\exp\left(2\Delta\Gamma +(q\log\log n+1)\left(1-\frac{2\Delta\Gamma}{\log n}\right)\right)\leq \Gamma (\log n)^{q}.
	\end{multline*}
	The `in particular' conclusion of Theorem~\ref{thm:aas} requires
	\begin{equation*}
		\lim_{n\to\infty}n^{\Gamma}\exp\left(-\frac{(\log n)^{qd-2}}{\Gamma}\right)=0,
	\end{equation*}
	which is satisfied, since $q>3/d$.
\end{proof}

\paragraph{Acknowledgements} The author would like to thank Vojt\v ech Kalu\v za for helpful discussions.

\bibliographystyle{plain}
\bibliography{citations}

\begin{thebibliography}{1}

\bibitem{BK1}
D.~Burago and B.~Kleiner.
\newblock Separated nets in {Euclidean} space and {Jacobians} of {biLipschitz}
  maps.
\newblock {\em Geometric and Functional Analysis}, 8:273--282, 1998.
\newblock \url{http://dx.doi.org/10.1007/s000390050056}.

\bibitem{DymondKaluza2019}
M.~Dymond and V.~Kalu{\v{z}}a.
\newblock Highly irregular separated nets.
\newblock {\em Israel Journal of Mathematics}, 253(2):501--554, 2023.

\bibitem{DKK2018}
M.~Dymond, V.~Kalu{\v{z}}a, and E.~Kopeck{\'a}.
\newblock {Mapping $n$ grid points onto a square forces an arbitrarily large
  Lipschitz constant}.
\newblock {\em Geometric and Functional Analysis}, 28(3):589--644, 2018.
\newblock \url{https://doi.org/10.1007/s00039-018-0445-z}.

\bibitem{Feige_approx}
U.~Feige.
\newblock Approximating the {B}andwidth via {V}olume {R}especting {E}mbeddings.
\newblock {\em Journal of Computer and System Sciences}, 60(3):510 -- 539,
  2000.

\bibitem{Grom}
M.~L. Gromov.
\newblock {\em Geometric Group Theory: Asymptotic invariants of infinite
  groups}.
\newblock London Mathematical Society lecture note series. Cambridge University
  Press, 1993.

\bibitem{Matousek_open}
J.~Matou\v{s}ek and A.~Naor~(eds.).
\newblock {Open problems on low-distortion embeddings of finite metric spaces},
  2011 (last revision).
\newblock {Available} at \url{kam.mff.cuni.cz/~matousek/metrop.ps}.

\bibitem{McM}
C.~T. McMullen.
\newblock Lipschitz maps and nets in {Euclidean} space.
\newblock {\em Geometric and Functional Analysis}, 8:304--314, 1998.
\newblock \url{http://dx.doi.org/10.1007/s000390050058}.

\bibitem{RY}
T.~Rivi\`ere and D.~Ye.
\newblock Resolutions of the prescribed volume form equation.
\newblock {\em Nonlinear Differential Equations and Applications},
  3(3):323--369, 1996.
\newblock \url{http://dx.doi.org/10.1007/BF01194070}.

\bibitem{romik2000stirling}
D.~Romik.
\newblock Stirling's approximation for n!: The ultimate short proof?
\newblock {\em The American Mathematical Monthly}, 107(6):556, 2000.

\end{thebibliography}

\end{document}